\documentclass[a4paper,12pt]{article}

\usepackage{amsmath}
\usepackage{amsfonts}
\usepackage{amssymb}
\usepackage{amsthm}
\usepackage{amsopn}
\usepackage{amscd}

\usepackage{pstricks}
\usepackage[utf8]{inputenc}
\usepackage{courier}
\usepackage{url}
\usepackage{color}

\usepackage[vcentering,dvips]{geometry}
\def\address#1#2{\begingroup
\noindent\parbox[t]{7.8cm}{%
\small{\scshape\ignorespaces#1}\par\vskip1ex
\noindent\small{\itshape E-mail address}%
\/: #2\par\vskip4ex}\hfill%
\endgroup}%

\author{Nathan Owen Ilten}
\title{Deformations of Smooth Toric Surfaces}
\date{}
\DeclareMathOperator{\spec}{Spec}
\DeclareMathOperator{\cone}{Cone}

\DeclareMathOperator{\tv}{TV}

\DeclareMathOperator{\conv}{Conv}
\DeclareMathOperator{\codim}{codim}

\DeclareMathOperator{\T}{\mathcal{T}}
\DeclareMathOperator{\aut}{Aut}

\DeclareMathOperator{\id}{id}
\DeclareMathOperator{\roof}{roof}

\newcommand{\CO}{\mathcal{O}}

\newtheorem{lemma}{Lemma}[section]
\newtheorem{prop}[lemma]{Proposition}
\newtheorem{cor}[lemma]{Corollary}
\newtheorem{thm}[lemma]{Theorem}

\theoremstyle{definition}
\newtheorem*{ex}{Example}
\newtheorem*{remark}{Remark}
\newtheorem*{defn}{Definition}

\begin{document}
\maketitle
\begin{abstract}
	For a complete, smooth toric variety $Y$, we describe the graded vector space $T_Y^1$. Furthermore, we show that smooth toric surfaces are unobstructed and that a smooth toric surface is rigid if and only if it is Fano. For a given toric surface we then construct homogeneous deformations by means of Minkowski decompositions of polyhedral subdivisions, compute their images under the Kodaira-Spencer map, and show that they span $T_Y^1$. 
\end{abstract}

\noindent Keywords: Toric varieties, deformation theory, T-varieties 
\\

\noindent MSC: Primary 14D15; Secondary 14M25.
\section*{Introduction}
The deformation theory of toric singularities has been rather well studied; see for example \cite{MR1329519}. It is possible to calculate the space of infinitesimal deformations, construct certain homogeneous deformations, and in some cases, construct a versal deformation all based on the combinatorics of the polyhedral cone corresponding to the toric singularity. On the other hand, with the exception of \cite{MR2092771}, \cite{MR2169828}, and some rigidity results there has, to our knowledge, been very little work done on deformations of complete toric varieties.

The goal of this paper is to start understanding the deformation theory of complete toric varieties. For simplicity's sake, we will assume that all varieties are smooth; we hope to be able to remove this assumption in a later paper. In section \ref{sec:T1} we will compute the vector space of infinitesimal deformations of an arbitrary smooth, complete toric variety. This essentially boils down to counting the number of connected components in certain graphs coming from hyperplane sections of the fan defining the toric variety. The formula takes a particularly simple form for toric surfaces. As an application we shall see that a complete, smooth toric surface is rigid if and only if it is Fano, as well as that weakly Fano varieties in higher dimensions are quite often rigid.

In section \ref{sec:homdefs}, we then concentrate on constructing homogeneous one-parameter deformations of complete, smooth toric surfaces by means of Minkowski decompositions of polyhedral subdivisions. This is analogous to the construction of deformations of toric singularities in terms of Minkowski decompositions of polytopes. For each such deformation we construct, we also compute its image under the Kodaira-Spencer map. This enables us to show that the homogeneous deformations we construct in fact form a basis for the space of infinitesimal deformations.\footnote{As pointed out by A. Mavlyutov, the construction of one-parameter
deformations as found in \cite{MR2092771} can be easily extended to construct
deformations spanning $T_Y^1$ as well.}

We assume that the reader is familiar with toric varieties as presented in \cite{MR1234037}. Throughout this paper, we will be using the following notation. As usual, $N$ will be a rank $n$ lattice, with dual lattice $M$. $N_\mathbb{Q}$ and $M_\mathbb{Q}$ shall denote the associated $\mathbb{Q}$-vector spaces. For any $M$-graded group $G$ and $u\in M$, let $G(u)$ be the grade $u$ component of $G$. $Y=\tv(\Sigma)$ will be a complete, smooth, $n$-dimensional toric variety corresponding to the complete smooth fan $\Sigma$ on $N_\mathbb{Q}$. By $\Sigma^{(k)}$ we denote the set of all $k$-dimensional cones in $\Sigma$. Let $\rho_1,\ldots,\rho_l\in\Sigma^{(1)}$ be all rays in the fan $\Sigma$. In the case $n=2$, we number these rays in some counterclockwise order; by $\rho_{l+1}$ we then mean $\rho_{1}$, etc. By $\nu(\rho_i)$ we shall denote the minimal generator of the ray $\rho_i$. Furthermore, by $D_i$ we denote the invariant divisor corresponding to $\rho_i$. 

\section{Infinitesimal Deformations}\label{sec:T1}
Let $Y$ be a smooth toric variety; we are interested in the vector space of infinitesimal deformations, which we term $T_Y^1$. Since $Y$ is smooth, we simply have $T_Y^1=H^1(Y,\T_Y)$.  Now this vector spaces carries an $M$ grading, so we can compute it by computing each homogeneous piece $T_Y^1(u)$. Furthermore, the cohomology of $\T_Y$ is closely intertwined with the cohomology of the boundary divisors of $Y$. We first shall show how to compute the cohomology of the boundary divisors and then proceed to use this to compute $T_Y^1$. We then use these computations to prove several rigidity results. 

\subsection{Cohomology of Boundary Divisors}\label{sec:cohbd}
Fix some smooth, complete variety $Y=\tv(\Sigma)$ and choose some weight $u\in M$; for $1\leq i \leq l$ we wish to calculate $H^1\left(Y,\CO(D_i)\right)(u)$. Using a result of Demazure, we show that this can be done by counting connected components in a certain graph. Let $\Gamma_i(u)$ be the graph with vertex set $V_i(u)=\left\{\nu(\rho_j),\ j\neq i\ |\ \langle \nu(\rho_j),u\rangle < 0 \right\}$ and two vertices $\nu(\rho_j)$ and $\nu(\rho_k)$ joined by an edge if and only if $\rho_j$ and $\rho_k$ are rays in some common cone of $\Sigma$.

\begin{prop}\label{prop:demazureresult}
	For any smooth, complete toric variety $Y$ we have $H^1\left(Y,\CO(D_i)\right)(u)=0$ if $\langle \nu(\rho_i),u\rangle \neq -1$.  Otherwise, $$\dim H^1\left(Y,\CO(D_i)\right)(u)= \max\{0,\dim H^0(\Gamma_i(u),\mathbb{C})-1\}.$$
\end{prop}
\begin{proof}
	Let $U_i(u)=\{v\in N_\mathbb{Q} | \langle v,u\rangle < h(v)\}$, where $h$ is the piecewise linear function on $\Sigma$ given by $h(\nu(\rho_i))=-1$, $h(\nu(\rho_j))=0$ for $j\neq i$. Then by \cite{MR0284446}, $H^p\left(Y,\CO(D_i)\right)(u)\cong H^p (N_\mathbb{Q},U_i(u))$ for all $p\geq 0$. Thus, we have the following exact sequence coming from relative cohomology:
	\begin{equation*}
		0\to H^0(Y,\CO(D_i))(u) \to \mathbb{C} \to H^0(U_i(u),\mathbb{C}) \to H^1(Y,\CO(D_i))(u)\to 0.
	\end{equation*}
	Suppose that $\langle \nu(\rho_i),u\rangle \neq -1$. If $u=0$ then $U_i(u)=\emptyset$ and thus $H^1(Y,\CO(D_i))(u)=0$. If instead $u\neq 0$, then we also have $H^1(Y,\CO(D_i))(u)=0$. Indeed, a direct calculation shows that $H^0(Y,\CO(D_i))(u)=0$. Furthermore, $U_i(\lambda u)$ is homeomorphic to $U_i(u)$ for all $\lambda \in \mathbb{N}$. Since $H^1(Y,\CO(D_i))$ is finite dimensional, it follows from the above exact sequence that $\dim H^0(U_i,\mathbb{C})(\lambda u)= \dim H^0(U_i,\mathbb{C})(u)=1$ and thus that $H^1(Y,\CO(D_i))(u)=0$.

	We can now assume that $\langle \nu(\rho_i),u\rangle = -1$. In this case, $H^0(U_i(u),\mathbb{C}) =H^0(\Gamma_i(u),\mathbb{C})$. Indeed, $U_i(u)$ can be retracted to  $\widetilde{U}_i(u)=U_i(u)\cap \roof(\Sigma)$, where $\roof(\Sigma)$ is the roof of the fan $\Sigma$. Now for each simplex $S$ of dimension larger than one in $\roof(\Sigma)$, $\widetilde{U}_i(u)\cap S$ can be replaced in $\widetilde{U}_i(u)$ with $\widetilde{U}_i(u)\cap \partial S$ without changing the connectivity of the set. Thus, we can replace $\widetilde{U}_i(u)$ by its intersection with the union of all one-simplices of $\roof(\Sigma)$. This is nothing other than $\Gamma_i(u)$.

	Now, one easily checks that $H^0(Y,\CO(D_i))=0$ unless $\Gamma_i(u)=\emptyset$. The formula then follows from the exact sequence.
\end{proof}

In dimensions two and three, we can calculate the first cohomology groups by counting connected components of an even simpler graph. Indeed, let $\Gamma_i^\circ(u)$ be the subgraph of $\Gamma_i(u)$ induced by those vertices whose corresponding rays share a common cone with $\rho_i$.
\begin{lemma}\label{lemma:nfisgood}
For $u\in M$ such that $\langle \nu(\rho_i),u\rangle =-1$, we have 
\begin{equation*}
	H^0(\Gamma_i(u),\mathbb{C})\leq H^0(\Gamma_i^\circ(u),\mathbb{C})
\end{equation*}
with equality holding if $n=\dim N\leq 3$.
\end{lemma}
\begin{proof}
	Any connected component of $\Gamma_i(u)$ has non-empty intersection with $\Gamma_i^\circ(u)$. Indeed, let $\widetilde{\Gamma}_i(u)$ be the graph with vertex set $V_i(u)\cup \{\nu(\rho_i)\}$ and edges induced by common cone membership as before. Then the vertex set of $\widetilde{\Gamma}_i(u)$ consists of exact those $\nu(\rho_j)$ such that $\langle \nu(\rho_j),u\rangle <0$. Thus, $\widetilde{\Gamma}_i(u)$ has the same number of connected components as a convex set, namely one. The inequality then follows.

	If $n=2$ the equality is immediately obvious. For $n=3$, we can consider $\Gamma_i(u)$ as a planar graph embedded in the plane $\langle \cdot, u \rangle = -1$ by identifying each $\nu(\rho_j))$ with the image of $\rho_j$ in this plane. If a ray $\rho_k$  sharing a common cone with $\rho_i$ does not intersect this plane, it is not in $\Gamma_i^\circ(u) \subset \Gamma_i(u)$. The two-dimensional cone generated by $\rho_i$ and $\rho_k$ intersects the above plane in a ray $\tau_k$, which no edges in $\Gamma_i(u)$ intersect. One easily sees that such $\tau_k$ partition both $\Gamma_i^\circ (u)$ and $\Gamma_i(u)$ into their connected components, so that the desired equality must hold.
\end{proof}

\subsection{The Generalized Euler Sequence}\label{sec:es}
We will now use our knowledge of the first cohomology groups of the boundary divisors to compute $T_Y^1$. Indeed, for a smooth, complete toric variety, the higher cohomology groups of the tangent bundle are isomorphic to the sum of those of the boundary divisors via the generalized Euler sequence:

\begin{lemma}
	\label{lemma:t1t2es}\cite{MR1272701}
	Let $Y$ be a smooth, complete toric variety with boundary divisors $D_1,\ldots,D_l$. Then $H^i(Y,\T_Y)\cong\bigoplus_{i=1}^l H^i(Y,\CO(D_i))$ as $M$-graded groups for $i\geq 1$.
\end{lemma}
\begin{proof}
		Taking the long exact cohomology sequence coming from the generalized Euler sequence 
		\begin{equation}    \label{eqn:es}           
			0 \mapsto N_1(Y) \otimes \CO_Y \to \bigoplus_{i=1}^l \CO(D_i) \to \T_Y \to 0 \,,\end{equation}
                and using the fact that $H^i(Y,N_1\otimes \CO_Y)$ vanishes for $i\geq 1$ gives us $$H^i(Y,\T_Y)\cong H^i(Y,\bigoplus_{i=1}^l \CO(D_i))\cong\bigoplus_{i=1}^l H^i(Y,\CO(D_i))$$ for $i\geq 1$. 
\end{proof}

Combining this with the cohomology computations of the previous section yields the following:
\begin{prop}\label{prop:t1}
For a smooth, complete toric variety $Y$, 
\begin{equation*}
	\dim T_Y^1(u)=\dim H^1(Y,\T_Y)(u)= \sum_{\langle \nu(\rho_i),u \rangle =-1} \max \{ 0,\dim H^0(\Gamma_i(u),\mathbb{C})-1\},
\end{equation*}
where if $\dim n=2,3$ we can replace the $\Gamma_i(u)$ with $\Gamma_i^\circ(u)$. 
\end{prop}

In the case of a toric surface, this becomes even more explicit:
\begin{cor}\label{cor:t1forsurfaces}
For a smooth, complete toric surface $Y=\tv(\Sigma)$
\begin{equation*}
\dim T_Y^1(u)=\#\left\{\rho_i\in\Sigma^{(1)}\left| \begin{array}{c}
		\langle \nu(\rho_i),u\rangle=-1\\
		 \langle \nu(\rho_{i\pm1}),u\rangle<0\end{array}\right.\right\}
\end{equation*}
for all $u\in M$. Furthermore, $T_Y^2=0$, i.e. $Y$ is unobstructed.
\end{cor}
\begin{proof}
	The first claim follows from the proposition and the fact that $\dim H^0(\Gamma_i^\circ(u),\mathbb{C})\leq 2$, with equality if and only if $\langle\nu(\rho_i),u\rangle=-1$ and $\langle \nu(\rho_{i\pm1}),u\rangle<0$. For the second, consider that $$H^2(Y,\T_Y)\cong \bigoplus H^2(Y,\CO(D_i))\cong \bigoplus H^0(Y,\CO(-D_i-\sum_{j=1}^l D_j))=0$$ by Serre duality.
\end{proof}

The following corollary is then immediate:
\begin{cor}
Let $Y$ be a smooth, complete toric surface. Let $\widetilde{Y}$ be an equivariant blow-up of $Y$. Then $T_{Y}^1\subset T_{\widetilde{Y}}^1$.
\end{cor}

\begin{figure}
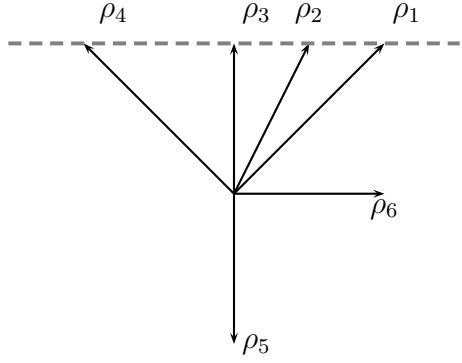

	\begin{center}
		\psset{unit=2cm}
	\pspicture(-1.5,-1)(1.5,1.3)
	\psline{->}(0, 1)
	\psline{->}(1, 0)
	\psline{->}(-1, 1)
	\psline{->}(1, 1)
	\psline{->}(0, -1)
	\psline{->}(.5, 1)
\rput(.5,1.2){$\rho_2$}
\rput(1.15,1.2){$\rho_1$}
\rput(0.15,1.2){$\rho_3$}
\rput(-.80,1.2){$\rho_4$}
\rput(0.15,-1){$\rho_5$}
\rput(1,-.1){$\rho_6$}
\psset{linecolor=gray,linestyle=dashed,linewidth=1.5pt}
\psline(-1.5,1)(1.5,1)
	\endpspicture
\end{center}
\caption{$\mathcal{F}_1$ with two successive blowups}\label{fig:ex1fan}	\end{figure}

\begin{ex}
	We consider the toric surface $Y$ whose fan $\Sigma$ has rays through  $(1,1)$, $(1,2)$, $(0,1)$, $(-1,1)$, $(0,-1)$, and $(1,0)$ as pictured in figure \ref{fig:ex1fan}. This is simply the first Hirzebruch $\mathcal{F}_1$ with two successive blowups. One easily checks that the only degrees $u$ for which $T_Y^1(u)$ isn't trivial are $[0,-1]$, $[-1,0]$, and $[1,-1]$. For $u=[0,-1]$, the dashed gray line in figure \ref{fig:ex1fan} marks the hyperplane $\langle \cdot, u\rangle = -1$. One thus sees that only the ray $\rho_3$ satisfies the conditions $\langle \nu(\rho_i),u\rangle =-1$, $\langle \nu(\rho_{i\pm 1}),u\rangle <0$ and thus, $\dim T_Y^1([0,-1])=1$. Likewise, for $u=[-1,0]$ and $u=[1,-1]$ the rays $\rho_1$ and $\rho_3$ respectively satisfy the conditions and so $\dim T_Y^1([-1,0])=1$ and $\dim T_Y^1([-1,1])=1$ as well.
\end{ex}

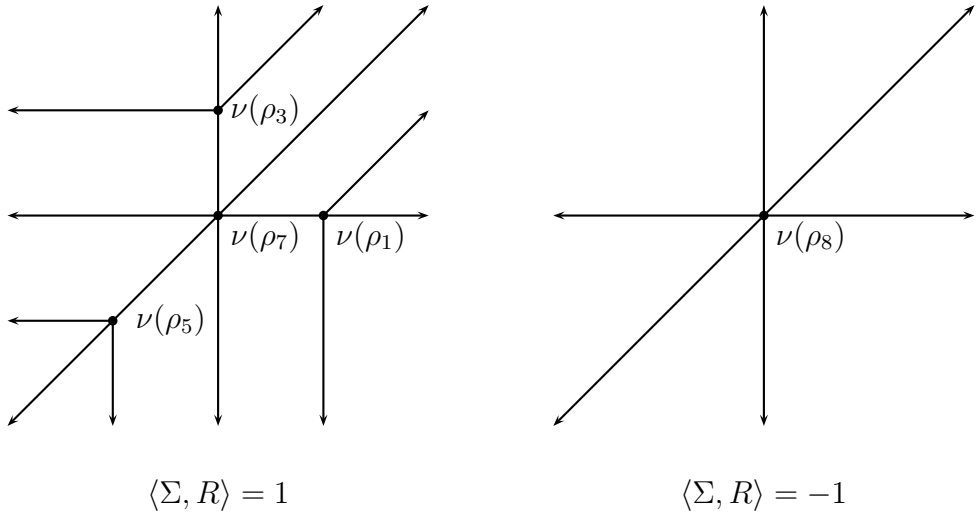
\begin{figure}
	\begin{center}
		\begin{displaymath}
			\begin{array}{c @{\qquad\qquad} c}
		\psset{unit=1.4cm}
	\begin{pspicture}(-2,-2)(2,2)
	\psline{->}(2,0)
	\psline{->}(0,2)
	\psline{->}(-2,-2)
	\psline{->}(-2,0)
	\psline{->}(2,2)
	\psline{->}(0,-2)
	\psline{->}(1,0)(1,-2)
	\psline{->}(-1,-1)(-1,-2)
	\psline{->}(-1,-1)(-2,-1)
	\psline{->}(0,1)(-2,1)
	\psline{->}(0,1)(1,2)
	\psline{->}(1,0)(2,1)
	\psdots(0,0)(-1,-1)(1,0)(0,1)
	\rput(1.45,-.2){$\nu(\rho_1)$}
	\rput(.45,1){$\nu(\rho_3)$}
	\rput(-.45,-1){$\nu(\rho_5)$}
	\rput(.45,-.2){$\nu(\rho_7)$}
	\end{pspicture}&

	\psset{unit=1.4cm}
	\begin{pspicture}(-2,-2)(2,2)
	\psline{->}(2,0)
	\psline{->}(0,2)
	\psline{->}(-2,-2)
	\psline{->}(-2,0)
	\psline{->}(2,2)
	\psline{->}(0,-2)
	\psdot(0,0)
	\rput(.45,-.2){$\nu(\rho_8)$}
	\end{pspicture}\\
	\\
	\langle \Sigma,R\rangle=1&\langle \Sigma,R\rangle=-1
\end{array}
\end{displaymath}
\end{center}
\caption{Hyperplane sections of three-dimensional fan}\label{fig:3dex}
\end{figure}

\begin{ex}
	We also consider a three-dimensional example. Let $\rho_1=\langle(1,0,1)\rangle$, $\rho_2=\langle(1,1,0)\rangle$, $\rho_3=\langle(0,1,1)\rangle$, $\rho_4=\langle(-1,0,0)\rangle$, $\rho_5\langle(-1,-1,1)\rangle$, $\rho_6=\rho_0=\langle(0,-1,0)\rangle$, $\rho_7=\langle(0,0,1)\rangle$, and $\rho_8=-\rho_7$. Now, let $\Sigma$ be the fan with top-dimensional cones generated by $\rho_i,\rho_{i+1},\rho_7$ or by $\rho_i,\rho_{i+1},\rho_8$ for $0\leq i <6$. If we set $-R=[0,0,-1]$, the affine slices $\langle \Sigma,R\rangle=\pm 1$ are pictured in figure \ref{fig:3dex}. We shall consider infinitesimal deformations of the threefold $Y=\tv(\Sigma)$.

	Note that the graph $\Gamma_7(-R)=\Gamma_7^\circ(-R)$ consists of the three non-connected vertices $\nu(\rho_1)$, $\nu(\rho_3)$, and $\nu(\rho_5)$. Thus, $\dim H^1(Y,\CO(D_7))(-R)=2$. One also easily checks that this is the only degree/ray combination for which the first cohomology doesn't vanish. Thus, $\dim T_Y^1=\dim T_Y^1(-R)=2$.
\end{ex}

\subsection{Rigidity Results}

In \cite{MR1420707} it was shown that a smooth, complete toric Fano variety is rigid. It is possible to use our explicit cohomology calculations to generalize this slightly; in many cases, it suffices to assume weakly Fano:
\begin{cor}
	Let $Y$ be a complete, smooth, weakly Fano toric variety, and assume that there is no equivariant embedding $\widetilde{A}_1\times (\mathbb{C}^*)^{n-2}\hookrightarrow Y$, where $\widetilde{A}_1$ is the minimal resolution of a toric $A_1$ singularity. Then $Y$ is rigid.
\end{cor}
	\begin{proof}
		Consider $u\in M$ and $1\leq i \leq l$ such that $\langle \nu(\rho_i),u\rangle=-1$ and take $\nu(\rho_j),\nu(\rho_k)\in\Gamma_i^\circ(u)$ with $j\neq k$. We shall show that $\nu(\rho_j)$ and $\nu(\rho_k)$ are in fact connected in $\Gamma_i(u)$; the claim then follows from corollary \ref{prop:t1} and lemma \ref{lemma:nfisgood}.

		Let $\gamma$ be the line segment connecting $\nu(\rho_j)$ and $\nu(\rho_k)$ and let $\widetilde{\gamma}$ be the projection of $\gamma$ to the roof of $\Sigma$. Now, since the roof of $\Sigma$ is concave,
		\begin{equation*}\langle v,u\rangle \leq \max \{\langle \nu(\rho_j),u\rangle,\langle \nu(\rho_k),u\rangle \}\leq -1\end{equation*}
			for all $v\in\widetilde{\gamma}$. Thus, if $\gamma$ doesn't intersect $\rho_i$, $\widetilde{\gamma}$ is in $U_i(u)$ and $\nu(\rho_j)$ and $\nu(\rho_k)$ are connected in $U_i(u)$ and thus also in $\Gamma_i(u)$. Suppose on the other hand that $\gamma$  intersects $\rho_i$, that is, that $\nu(\rho_i)\in\widetilde{\gamma}$. Then from the concavity of the roof and $\langle \nu(\rho_i),u\rangle=-1$, it follows that  $\langle \nu(\rho_j),u\rangle=\langle \nu(\rho_k),u\rangle=-1$. Since however both $\rho_j$ and $\rho_k$ share common cones with $\rho_i$, one easily sees that the subfan of $\Sigma$ with rays $\rho_i$, $\rho_j$, and $\rho_k$ corresponds to the toric variety $\widetilde{A}_1\times (\mathbb{C}^*)^{n-2}$, a contradiction.
	\end{proof}

Now, any weakly Fano smooth, complete toric surface which isn't Fano does in fact admit an embedding of $\widetilde{A}_1$, so the above result doesn't provided anything new for $n=2$. However, we can show that non-Fano surfaces are in fact never rigid:
\begin{cor}
A smooth, complete toric surface $Y$ is rigid if and only if $Y$ is Fano. 
\end{cor}
\begin{proof}
	Every smooth, complete toric surface $Y$ can be constructed as an equivariant blow-up of $\mathbb{P}^2$, $\mathbb{P}^1\times \mathbb{P}^1$, or the Hirzebruch surface $\mathcal{F}_r$, cf. \cite{MR1234037} page 43. Of these varieties, an application of corollary \ref{cor:t1forsurfaces} shows that only $\mathbb{P}^2$, $\mathbb{P}^1\times \mathbb{P}^1$, and $\mathcal{F}_1$ are rigid. One then checks case by case rigidity for blow-ups of these surfaces. Once a blow-up is no longer rigid, the above corollary tells us that further blow ups are also non-rigid; it thus turns out that we only have to check finitely many cases and that the rigid surfaces are exactly those which are Fano.
\end{proof}

\begin{ex}
	Consider rays $\rho_0\ldots\rho_4$ generated by $(0,0,-1)$, $(0,0,1)$, $(1,0,1)$, $(1,1,1)$, and  $(0,1,1)$, respectively. Let $\Sigma$ be the fan with top-dimensional cones spanned by $\rho_1,\rho_2,\rho_3$, $\rho_1,\rho_3,\rho_4$, $\rho_0,\rho_1,\rho_4$, and $\rho_0,\rho_i,\rho_{i+1}$ for $i=1,2,3$. Then the complete toric threefold $Y=\tv(\Sigma)$ is rigid. Indeed, although it isn't Fano, it is weakly Fano, and one easily checks that there is no equivariant embedding of $\widetilde{A}_1\times (\mathbb{C}^*)$.
\end{ex}

\section{Homogeneous One-Parameter Deformations}\label{sec:homdefs}
In \cite{MR1329519}, Altmann introduced so-called toric deformations of affine toric varieties to study the deformation theory of toric singularities. We wish to construct an analogon in the global setting.

Recall that a deformation $\pi:X\to S$ of an affine toric variety $Y$ is called toric if ${X}$ is a toric variety, and the natural inclusion $Y\hookrightarrow X$ is torus equivariant and induces an isomorphism on the closed orbits. Let $Y$ correspond to a cone $\sigma$ in $N_\mathbb{Q}$ and $X$ to a cone  $\widetilde\sigma$ in $\widetilde{N}_\mathbb{Q}$ with $\codim(Y,X)=1$. It turns out that the inclusion $Y\hookrightarrow X$ is given by a single binomial $\chi^{s^1}-\chi^{s^2}$, with $s^1,s^2\in\widetilde{\sigma}^\vee\cap\widetilde{M}$; let $R$ be the common image of $s^1,s^2$ in $\sigma^\vee\cap M\setminus 0$. Setting the parameter $t=\chi^{s^1}-\chi^{s^2}$ thus gives a natural map $X\to\mathbb{A}^1$, which is a $1$-parameter deformation of $Y$; we say that it has degree $-R$. This deformation is not necessarily equal to the deformation $\pi$, but they share certain properties.

Ideally, in the global setting we would like to consider deformations $\pi:X\to S$ of a complete, smooth toric surface $Y$ which are locally one-parameter toric deformations in a fixed degree $-R$.\footnote{For $Y$ non-complete and non-smooth, the  author employed exactly such a construction in \cite{ilten08} to describe certain simultaneous resolutions.} However, this is simply not possible, since we can never find an $R\neq0$ living in all $\sigma^\vee$, $\sigma\in\Sigma$. Instead, we will consider deformations which are locally one-parameter toric deformations of degree $-R$ only on those open sets $\tv(\sigma)$ with $R\in\sigma^\vee$. We first describe how to construct such deformations and then compute their images under the Kodaira-Spencer map.

\subsection{Construction}
Let $Y=\tv(\Sigma)$ be a smooth, complete toric surface and let $R\in M$ with $R$ primitive. We can choose a basis of $N$ such that $R=[0,1]$ in the corresponding basis of $M$. We then consider the polyhedral subdivision 
\begin{align*}
\Xi_0&=\Sigma\cap\{v \in N_\mathbb{Q}|\langle v,[0,1]\rangle=1\}.\\
\end{align*}
Let $\mathbb{L}=\{v \in N|\langle v,[0,1]\rangle=1\}$ be the lattice attained when considering the point $(0,1)\in N$ as the origin. We can thus consider $\Xi_0$ to be a polyhedral subdivision of $\mathbb{Q}$. Let $\Delta_0^0,\ldots,\Delta_0^{m+1}$ be the line segments in $\Xi_0$, ordered such that $\Delta_0^i\geq \Delta_0^{i+1}$. We choose our ordering of the rays $\rho_i\in\Sigma^{(1)}$ such that $\rho_1$ passes between $\Delta_0^0$ and $\Delta_0^1$.

\begin{defn}
A subdivision decomposition $(\widetilde{\Xi}_0,\widetilde{\Xi}_t)$ of $\Xi_0$ consists of the sets \\$\widetilde{\Xi}_0 = \{\widetilde{\Delta}_0^0,\ldots,\widetilde{\Delta}_0^{m+1}\}$ and $\widetilde{\Xi}_t=\{\widetilde{\Delta}_t^0,\ldots,\widetilde{\Delta}_t^{m+1}\}$ such that
\begin{enumerate}
\item $\widetilde{\Delta}_0^i$, $\widetilde{\Delta}_t^i$ are polytopes in $\mathbb{Q}$;
\item $\widetilde{\Delta}_0^i\geq\widetilde{\Delta}_0^{i+1}$, $\widetilde{\Delta}_t^i\geq\widetilde{\Delta}_t^{i+1}$ for all $i$ and both $\widetilde{\Xi}_0$ and $\widetilde{\Xi}_t$ cover $\mathbb{Q}$;
\item $\Delta_0^i=\widetilde{\Delta}_0^i+\widetilde{\Delta}_t^i$ for all $i$.
\end{enumerate} 
Furthermore, we say that $(\widetilde{\Xi}_0,\widetilde{\Xi}_t)$ is admissible if for each $1\leq i\leq m$ either $\widetilde{\Delta}_0^i$ or $\widetilde{\Delta}_t^i$ is a lattice point.
\end{defn}

\begin{remark}
	To each admissible subdivision decomposition $(\widetilde{\Xi}_0,\widetilde{\Xi}_t)$  we can associate the $m+2$-tuples $a=(a_0,a_1,\ldots,a_m,a_{m+1})$ and $\lambda=(\lambda_0,\ldots,\lambda_{m+1})$ with $a_i\in\{-1,1\}$ and $\lambda_i\in\mathbb{Z}$ as follows: For $0\leq i \leq m+1$ we choose $\lambda_i\in\mathbb{Z}$ such that ${\Delta}_0^i=\widetilde{\Delta}_0^i+\lambda_i$ if possible, in which case we set $a_i=1$, otherwise $a_i=-1$ and we choose $\lambda_i\in\mathbb{Z}$ such that ${\Delta}_0^i=\widetilde{\Delta}_t^i+\lambda_i$. Note that $\lambda_i=\lambda_{i+1}$ if $a_i=a_{i+1}$ and $\lambda_i+\lambda_{i+1}=\Delta_0^i\cap\Delta_0^{i+1}$ if $a_i\neq a_{i+1}$. 

Conversely, an integer $\lambda_0$ and a binary $m+2$-tuple $a$ define a subdivision decomposition of $\Xi_0$. Such a decomposition is admissible if and only if $\Delta_0^i\cap\Delta_0^{i+1}$ is a lattice point for every $i$ such that $a_i\neq a_{i+1}$. 
\end{remark}

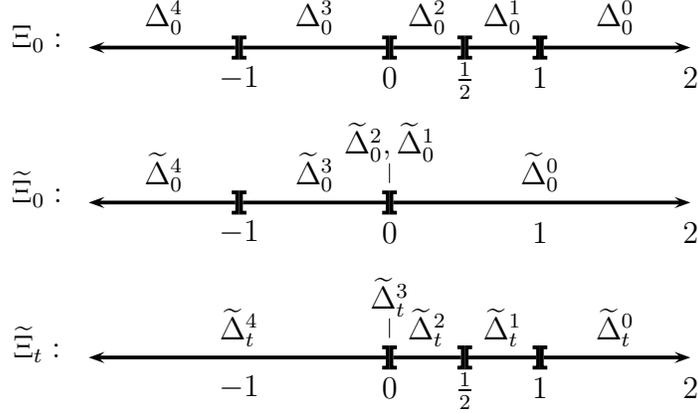
\begin{figure}
	\begin{center}
		\begin{displaymath}
		\begin{array}{c c}
\Xi_0:&
		\psset{unit=2cm,linewidth=.05cm}
		\begin{pspicture}(-2,0)(2,.7)
\psline{<-]}(-2,0)(-1,0)
\psline{[-]}(-1,0)(0,0)
\psline{[-]}(0,0)(.5,0)
\psline{[-]}(.5,0)(1,0)
\psline{[->}(1,0)(2,0)
\rput(0,-.2){$0$}
\rput(.5,-.22){$\frac{1}{2}$}
\rput(1,-.2){$1$}
\rput(2,-.2){$2$}
\rput(-1,-.2){$-1$}
\rput(1.5,.2){$\Delta_0^0$}
\rput(.75,.2){$\Delta_0^1$}
\rput(0.25,.2){$\Delta_0^2$}
\rput(-.5,.2){$\Delta_0^3$}
\rput(-1.5,.2){$\Delta_0^4$}
\end{pspicture} \\
\\
\widetilde{\Xi}_0:&
		\psset{unit=2cm,linewidth=.05cm}
		\begin{pspicture}(-2,0)(2,.7)
\psline{<-]}(-2,0)(-1,0)
\psline{[-]}(-1,0)(0,0)
\psline{[->}(0,0)(2,0)
\rput(0,-.2){$0$}
\rput(1,-.2){$1$}
\rput(2,-.2){$2$}
\rput(-1,-.2){$-1$}
\rput(1,.2){$\widetilde{\Delta}_0^0$}
\rput(-.5,.2){$\widetilde{\Delta}_0^3$}
\rput(-1.5,.2){$\widetilde{\Delta}_0^4$}
\rput(0,.4){$\widetilde{\Delta}_0^2,\widetilde{\Delta}_0^1$}
\psline[linewidth=.01cm](0,.26)(0,.13)
\end{pspicture} \\
\\
\widetilde{\Xi}_t:&
\psset{unit=2cm,linewidth=.05cm}
		\begin{pspicture}(-2,0)(2,.7)
\psline{<-]}(-2,0)(0,0)
\psline{[-]}(.5,0)(1,0)
\psline{[-]}(0,0)(.5,0)
\psline{[->}(1,0)(2,0)
\rput(0,-.2){$0$}
\rput(1,-.2){$1$}
\rput(2,-.2){$2$}
\rput(-1,-.2){$-1$}
\rput(.5,-.22){$\frac{1}{2}$}
\rput(1.5,.2){$\widetilde{\Delta}_t^0$}
\rput(.75,.2){$\widetilde{\Delta}_t^1$}
\rput(0.25,.2){$\widetilde{\Delta}_t^2$}
\psline[linewidth=.01cm](0,.26)(0,.13)
\rput(0,.4){$\widetilde{\Delta}_t^3$}
\rput(-1,.2){$\widetilde{\Delta}_t^4$}
\end{pspicture} \\
\end{array}
\end{displaymath}
	\end{center}
	\caption{A possible subdivision decomposition}\label{fig:ex1subdv}
\end{figure}

\begin{ex}
	We continue the example from figure \ref{fig:ex1fan} in section \ref{sec:es}. If we take the standard basis, then we attain the polyhedral subdivision induced by $\Sigma$ on the dashed gray line in figure \ref{fig:ex1fan}. We picture this subdivision $\Xi_0$  in figure \ref{fig:ex1subdv}. A possible subdivision decomposition is pictured as well; in this decomposition, $\widetilde{\Delta}_0^1$, $\widetilde{\Delta}_0^2$, and $\widetilde{\Delta}_t^3$ are all equal to the lattice point $0$. For this decomposition, we have $(a_0,\ldots,a_4)=(1,-1,-1,1,1)$ and $(\lambda_0,\ldots,\lambda_4)=(1,0,0,0,0)$. Note that modulo the choice of $\lambda_0$, there are exactly four possible decompositions, corresponding to the tuples $a=(1,1,1,1,1)$, $a=(1,-1,-1,-1,1)$, $a=(1,-1,-1,1,1)$, and $a=(1,1,1,-1,1)$. 
\end{ex}

From each admissible decomposition $(\widetilde{\Xi}_{0},\widetilde{\Xi}_{t})$ of $\Xi_0$ we will construct a one-parameter deformation of $Y$. We will do this locally on $\tv(\sigma)$ for each $\sigma\in\Sigma^{(2)}$ and then glue together to attain a global family. We first set up some notation. Let $\sigma_i$ be the cone in $\Sigma$ spanned by $\rho_i,\rho_{i+1}$. Note that for $1\leq i \leq m$, $\cone(\Delta_0^i)=\sigma_i$. For each $1\leq i\leq l$, let $w_i^1=[r_i^1,s_i^1],w_i^2=[r_i^2,s_i^2]$ be positively oriented minimal lattice generators of $\sigma_i^\vee$ and set $Z_{i,j}=x^{r_i^j}y^{s_i^j}$ for $j=1,2$. Then $\tv(\sigma_i)=\spec A_i$, where $$A_i=\mathbb {C}[\sigma_i^\vee\cap M]=\mathbb{C}[Z_{i,1},Z_{i,2}].$$ Note furthermore that $w_i^1=-w_{i+1}^2$.

Let $a$ and $\lambda$ be as in the above remark. For $i>m+1$ or $i<0$ we set $a_i=1$ and $\lambda_i=0$. Then for $0\leq i \leq l-1$ and $j\in\{1,2\}$ define polynomials
\begin{align*}
	\widetilde{Z}_{i,j}=\left\{\begin{array}{l l}
		x^{r_i^j}y^{s_i^j+\lambda_i r_i^j}(y-t)^{-\lambda_i r_i^j}\quad a_i=1\\
		x^{r_i^j}y^{-\lambda_i r_i^j}(y-t)^{s_i^j+\lambda_i r_i^j}\quad a_i=-1\end{array}\right.
\end{align*}
Likewise, define rings
\begin{align*}
	\widetilde{A}_i&=\mathbb{C}[t,\widetilde{Z}_{i,1},\widetilde{Z}_{i,2}] \qquad \textrm{if}\quad 0\leq i \leq m+1;\\
	\widetilde{A}_i&=\mathbb{C}[t,\widetilde{Z}_{i,1},\widetilde{Z}_{i,2},y(y-t)^{-1}] \qquad \textrm{if}\quad i> m+1.\\
\end{align*}
Then the natural map $\pi_i:\spec \widetilde{A}_i\to\spec \mathbb{C}[t]$ defines a one-parameter deformation of $\spec A_i=\tv(\sigma_i)$.
The maps $\pi_i$ glue naturally to give a map $\pi:X\to \mathbb{A}^1$. Indeed, for some $0\leq i,j< l$ identify $\spec \widetilde{A}_i$ with $\spec \widetilde{A}_j$ on all open subsets $\spec B$, where $B$ is a localization of both $\widetilde{A}_i$ and $\widetilde{A}_j$.
\begin{thm}
	The family $\pi:X\to \mathbb{A}^1$ is a one-parameter deformation of $Y=\tv(\Sigma)$.
\end{thm}
	\begin{proof}
		We just need to show that $\pi^{-1}(0)=\tv(\Sigma)$; this amounts to showing that the gluing of the $\spec \widetilde{A}_i$ restricted to the special fiber is equivalent to the gluing of the $\spec A_i$ present in $\tv(\Sigma)$. Now, it is immediate that any one of the gluings of the $\spec \widetilde{A}_i$ restricted to the special fiber induces a gluing on the $\spec A_i$ already present in $\tv(\Sigma)$. Thus, we just need to show that any gluing between the $A_i$ lifts to a gluing between the $\widetilde{A}_i$.

	Each ray $\rho_i$ corresponds to the gluing induced by the diagram
	\begin{equation*}
		A_{i-1}\to ({A_{i-1}})_{Z_{i-1,1}}=({A_{i}})_{Z_{i,2}}  \leftarrow A_{i}.
	\end{equation*}
	For $i\neq 0,m+2$ this gluing can be lifted to that induced by the diagram 
	\begin{equation*}
		\widetilde{A}_{i-1}\to ({\widetilde{A}_{i-1}})_{\widetilde{Z}_{i-1,1}}=({\widetilde{A}_{i}})_{\widetilde{Z}_{i,2}}  \leftarrow \widetilde{A}_{i}.
	\end{equation*}
	Indeed, the equality in the middle follows from the fact that $\widetilde{Z}_{i-1,1}=\widetilde{Z}_{i,2}^{-1}$ for $i\neq 0,m+2$.

On the other hand, for $i=m+2,l$ this gluing can be lifted to that induced by the diagram 
\begin{equation*}
	\widetilde{A}_{i-1}\to ({\widetilde{A}_{i-1}})_{\widetilde{Z}_{i-1,1},y^{-1}(y-t)}=({\widetilde{A}_{i}})_{\widetilde{Z}_{i,2},y^{-1}(y-t)}  \leftarrow \widetilde{A}_{i};
	\end{equation*}
The equality in the middle is easily checked. All other gluings in $\tv(\Sigma)$ are induced by the above.
\end{proof}

\begin{remark}
Let $\Sigma_+=\{\sigma\in\Sigma | R\in(\sigma^\vee)^\circ\}$. The deformation $\pi$ restricted to the open subset $\tv(\Sigma_+)\subset Y$ can then be completely described by a fan in a three-dimensional lattice $\widetilde{N}=\mathbb{Z}^3$. Indeed, for $1\leq i\leq m$ set $$\widetilde{\sigma}_i=\conv\{(\widetilde{\Delta}_0^i,1,0),(\widetilde{\Delta}_t^i,0,1)\}\subset\widetilde{N}_\mathbb{Q}.$$ The cones $\widetilde{\sigma}_i$ fit together to induce a fan $\widetilde{\Sigma}_+$ in $N_\mathbb{Q}$. Let $\pi_+:\tv(\widetilde{\Sigma}_+)\to\mathbb{A}^1$ be the map given by $t=\chi^{[0,1,0]}-\chi^{[0,0,1]}$ with $t$ the parameter of $\mathbb{A}^1$. Then $\pi_+$ is a deformation of $\tv(\Sigma_+)$ which is a toric deformation on each $\tv(\sigma_i)$; this follows from the local description in \cite{MR1329519}. Furthermore, calculation shows that it is simply the restriction of $\pi$ to $\tv(\Sigma_+)$. Note that the inclusion of the special fiber $\tv(\Sigma_+)\hookrightarrow\tv(\widetilde{\Sigma}_+)$ is induced by the unique map $N\to \widetilde{N}$ sending $v\in\mathbb{L}$ to $(v,1,1)$. 
\end{remark}

\begin{remark}
The construction of the deformation $\pi:X\to\mathbb{A}^1$ may appear somewhat technical; the same deformation can be described quite elegantly using the theory of T-varieties \cite{MR2426131}. Indeed, if we set $\Xi_\infty=\Sigma\cap\{v \in N_\mathbb{Q}|\langle v,[0,1]\rangle=-1\}$ (with proper labeling) and let $\Xi=\Xi_0\otimes \{0\}+\Xi_\infty\otimes\{\infty\}$ be a divisorial fan on $\mathbb{P}^1$, $Y=\tv(\Sigma)$ is equal to the T-variety associated to $\Xi$. Now, consider the divisors $D_0=V(y)$, $D_t=V(y-t)$, and $D_\infty=V(y^{-1})$ on $\mathbb{P}^1\times\mathbb{A}^1$ with $t$ the natural $\mathbb{A}^1$ coordinate and $y$ the natural $\mathbb{P}^1$ coordinate. The divisorial fan $\widetilde{\Xi}=\widetilde{\Xi}_0\otimes D_{0}+\widetilde{\Xi}_t\otimes D_{t}+\widetilde{\Xi}_\infty\otimes D_{\infty}$ corresponds to a T-variety equal to the total space $X$. The map $\pi$ corresponds to the projection onto the $\mathbb{A}^1$ factor. 

The local theory of deformations of T-varieties is being developed by Robert Vollmert in \cite{vollmert09}. Hendrik Süß has also constructed an example of a smoothing of a Fano threefold using the language of T-varieties \cite{suess08}. The author of the present paper plans an upcoming paper outlining the general theory of deformations of (not necessarily smooth) complete T-varieties.

The above description of the deformation $\pi$ has the further advantage that the fibers $\pi^{-1}(t)$ for $t\neq 0$ can be identified with the T-variety over $\mathbb{P}^1$ coming from the divisorial fan with coefficients $\widetilde{\Xi}_0$ at $0$, $\widetilde{\Xi}_t$ at $t$, and $\widetilde{\Xi}_\infty$ at $\infty$.
\end{remark}

\subsection{The Kodaira-Spencer Map}
Let $\pi:X\to \mathbb{A}^1$ be the deformation coming from the admissible decomposition $(\widetilde{\Xi}_0,\widetilde{\Xi}_t)$ or equivalently the integer $\lambda_0$ and admissible binary $m+2$-tuple $a$. As above, For $l>i>m+1$ we set $a_i=1$ and $\lambda_i=0$.
Note that when we compute the image of the Kodaira-Spencer map, we will be expressing cohomology classes as \v{C}ech cocycles with respect to the natural affine invariant open covering $\mathfrak{U}=\left\{\tv(\tau) | \tau\in\Sigma^{(2)}\right\}$; to describe such a cocycle, it will be in fact sufficient to give sections $d_{i-1,i}\in\Gamma(TV(\rho_i),\T_Y)$ for $1\leq i \leq l$.

\begin{thm}
The image of $\pi$ in $H^1(Y,\T_Y)$ can be described as the \v{C}ech cocycle induced by $d_{i-1,i}\in\Gamma(\tv(\rho_i),\T_Y)$, where
\begin{equation*}
d_{i-1,i}=\left\{\begin{array}{l l}
0 & i\notin\{m+2,0\} \ \mathrm{and}\ a_i=a_{i-1}\\
{a_{i-1}}(\lambda_{i-1}-\lambda_i)xy^{-1}\frac{\partial}{\partial x} & i\in\{m+2,0\}\ \mathrm{and}\ a_i=a_{i-1}.\\
{a_{i-1}}( (\lambda_i+\lambda_{i-1}) xy^{-1}\frac{\partial}{\partial x}+\frac{\partial}{\partial y})& a_i\neq a_{i-1}\\
\end{array}\right.
\end{equation*}
In particular, the deformation is homogeneous of degree $[0,-1]=-R$. 
\end{thm}
\begin{proof}
	We proceed to compute the image of $\pi$ under the Kodaira-Spencer map as described in \cite{MR2247603}. We shall thus be considering $\pi$ as a deformation over $\spec \mathbb{C}[t]/t^2$ instead of over $\spec \mathbb{C}[t]$.
We first claim that $\pi$ is trivial on the sets $\tv(\sigma_i)$ for $\sigma_i\in\Sigma^{(2)}$, that is, there are isomorphisms 
\begin{equation*}
	\theta_i^{\#}:\widetilde{A}_i\otimes \mathbb{C}[t]/t^2\to A_i \otimes \mathbb{C}[t]/t^2
\end{equation*}
respecting the $\mathbb{C}[t]/t^2$ structure. Indeed, define $\theta_i^{\#}$ by 
\begin{equation*}
	\theta_i^{\#}:\widetilde{Z}_{i,k}\mapsto Z_{i,k}
\end{equation*}
for $k=1,2$. This uniquely defines $\theta_i^{\#}$ since for $m+1<i<0$, $y(y-t)^{-1}\in \mathbb{C}[t,\widetilde{Z}_{i,1},\widetilde{Z}_{i,2}]/t^2$. Indeed, $y(y-t)^{-1}=1+ty^{-1}$ modulo $t^2$. One easily checks that  $\theta_i^{\#}$ is an isomorphism.

Now we set $\theta_{i,j}^{\#}:=\theta_j^{\#}(\theta_{i}^{\#})^{-1}$; this corresponds to an automorphism $$\theta_{i,j}\in\aut\left(\tv(\sigma_i\cap\sigma_j)\times\spec \mathbb{C}[t]/t^2\right).$$ To compute the image of the Kodaira-Spencer map, it will be sufficient to only consider the automorphisms $\theta_{i-1,i}$ that is, those corresponding to $\theta_{i-1, i}{^\#}$. We actually want to calculate for each $i$ the derivation $d_{i-1,i}$ defined by
\begin{equation*}
\theta_{i-1, i}^\#=\id+t\cdot d_{i-1,i}.
\end{equation*}

It's time to start calculating. We first note that $$\frac{y-t}{y}=\left(1-{a_i}\widetilde{Z}_{i,1}^{r_i^2}\widetilde{Z}_{i,2}^{-r_i^1}t\right)^{{a_i}}.$$ Now, let $b_1^k,b_2^k$ be such that 
$w_{i-1}^k=b_1^k w_i^1+b_2^k w_i^2$. We then have
\begin{align*}
	\theta_{i-1, i}^\#(Z_{i-1,k})&=\theta_i^\#\left(\widetilde{Z}_{i-1,k}\right)=\theta_i^\#\left(\widetilde{Z}_{i,1}^{b_1^k}\widetilde{Z}_{i,2}^{b_2^k}\cdot
\left(\frac{y-t}{y}\right)^{\frac{1}{2}(a_i-a_{i-1})s_{i-1}^k+(a_i\lambda_i-a_{i-1}\lambda_{i-1})r_{i-1}^k}\right)\\
&={Z}_{i-1,k}\cdot
\left(1-a_iy^{-1}t\right)^{\frac{1}{2}a_i(a_i-a_{i-1})s_{i-1}^k+a_i(a_i\lambda_i-a_{i-1}\lambda_{i-1})r_{i-1}^k}
\end{align*}
Thus, we have that
\begin{align*}
d_{i-1,i}(Z_{i-1,k})=-\frac{1}{2}(a_i-a_{i-1})s_{i-1}^k\cdot y^{-1}Z_{i-1,k}-(a_i\lambda_i-a_{i-1}\lambda_{i-1})r_{i-1}^k\cdot y^{-1}Z_{i-1,k},
\end{align*}
that is,
\begin{align*}
d_{i-1,i}=\frac{1}{2}(a_{i-1}-a_i)\frac{\partial}{\partial y}+(a_{i-1}\lambda_{i-1}-a_i\lambda_i)xy^{-1}\frac{\partial}{\partial x}.
\end{align*}
The expressions for $d_{i-1,i}$ in the statement of the theorem then follow from the facts that $a_{i-1},a_i\in\{-1,1\}$ and that for $i\notin\{0,m-2\}$, the equality $a_i=a_{i-1}$ implies that $\lambda_{i-1}=\lambda_i$.
\end{proof}

\begin{remark}
The cocycle condition tells us that we should have $\sum_{i=1}^l d_{i-1,i}=0$. This amounts to the equalities
\begin{align*}
\sum_{i=1}^l (a_{i-1}-a_i)&=0\\
\sum_{i=1}^l (a_{i-1}\lambda_{i-1}-a_i\lambda_i)&=0
\end{align*}
which are easily seen to hold.
\end{remark}

 The following corollary tells us how the image of the Kodaira-Spencer map looks after applying the isomorphism from the Euler sequence.

\begin{cor}
	The image of $\pi$ in $\bigoplus_{i=1}^l H^1(Y,\CO(D_i))$ under the isomorphism\\ $\bigoplus H^1(Y,\CO(D_i))\cong H^1(Y,\T_Y)$ can be described as the \v{C}ech cocycle induced by $$g_{j-1,j}\in\bigoplus_{i=1}^l \Gamma(\tv(\rho_j),\CO(D_i))\cdot e_{D_i},$$ where:

\begin{align*}
g_{m+1,m+2}&=\sum_{\substack{1\leq j\leq m+1\\ a_j\neq a_{j+1} } } -({a_{j-1}} y^{-1})\cdot e_{D_j};\\
	g_{j-1,j}&= ({a_{j-1}} y^{-1})\cdot e_{D_j}\quad \mathrm{for}\ j\notin\{m+2,0\}\ \mathrm{and}\ a_i\neq a_{i-1};\\
g_{j-1,j}&=0\quad \mathrm{otherwise}.
\end{align*}
\end{cor}
\begin{proof}
We first recall that the map $\CO(D_i)\to \T_Y$ is locally given by
	\begin{equation*}
				\chi^u \mapsto \left(\nu(\rho_i)_1 x \frac{\partial}{\partial x}+\nu(\rho_i)_2 y \frac{\partial}{\partial y}\right)\cdot \chi^u
			\end{equation*}
			for $u\in M$.
			Suppose now that $m+2=l$. In this case, one easily checks for $j\neq m+2$ that the image of $g_{j-1,j}$ is exactly $d_{j-1,j}$ as in the above theorem. It follows that the image of $g_{m+1,m+2}$ is $d_{m+1,m+2}$ as above. Indeed, $g_{m+1,m+2}=-\sum_{j\neq m+2} g_{j-1,j}$ and $d_{m+1,m+2}=-\sum_{j\neq m+2} d_{j-1,j}$. 

			Suppose instead that $l>m+2$. For $m+2\leq i \leq l-1$ define $f_i=\sum_k\alpha_k y^{-1} e_{D_k}$, and for other $i$ set $f_i=0$. For each $i$ we have $f_i\in\Gamma (\tv(\sigma_i), \bigoplus \CO(D_j))$ and thus this defines a homogeneous degree $-R$ element in the zeroth term of the \v{C}ech complex for $\bigoplus \CO(D_j)$. The image of $f$ in the first term of the \v{C}ech complex is the cocycle induced by $f_{m+1,m+2}= \sum_k \alpha_k y^{-1} e_{D_k}$, $f_{l-1,l}=-f_{m+1,m+2}$. Now, it is not difficult to choose numbers $\alpha_k$ such that the image of $g_{l-1,l}+f_{l-1,l}$ is $d_{l-1,l}$ as above. Similar to the previous paragraph it then also follows that the image of $g_{m+1,m+2}+f_{m+1,m+2}$ is $d_{m+1,m+2}$. Since however the cocycle induced by the $f_{j-1,j}$ lies in the image of the first term of the \v{C}ech complex, it is cohomologous to $0$.
\end{proof}

Fix now some $R\in M$ primitive. For $2\leq i\leq m$ with $\langle \nu(\rho_i),R\rangle=1$, let $\pi(i)$ be the deformation corresponding to $\lambda_0=0$ and $a_j=1$ for $j<i$ and $a_j=-1$ for $j\geq i$.

\begin{cor}
	The deformations $\{\pi(i)\}_{\substack { 2\leq i \leq m\\ \langle \nu(\rho_i),R\rangle=1 }}$ form a basis of $T_Y^1(-R)$.
\end{cor}
	\begin{proof}
		From the above corollary we know that the image of $\pi(i)$ in $\bigoplus_{j=1}^l H^1(Y,\CO(D_j))$ is the \v{C}ech cocycle induced by $g_{i-1,i}=y^{-1}\cdot e_{D_i}$ and $g_{m+1,m+2}=-y^{-1}\cdot e_{D_i}$. This cocycle is not cohomologous to $0$. Indeed, since $\Gamma(\tv(\sigma_i),\CO(D_i))(-R)=\Gamma(\tv(\sigma_{i-1}),\CO(D_i))(-R)=0$, there is no element in the zeroth term of the \v{C}ech complex of $\CO(D_i)$ whose image is the above cocycle.
We have thus seen that for each $2\leq i \leq m$ with $\langle \nu(\rho_i), R \rangle =1$ that the image of $\pi(i)$ is a non-trivial cocycle. Clearly these cocycles are linearly independent, since each is supported on a different bundle $\CO(D_i)$. Finally, from corollary \ref{cor:t1forsurfaces} we see that they must span $T_Y^1(-R)$.
	\end{proof}

\begin{figure}
	\begin{center}
		\begin{displaymath}
		\begin{array}{r c}
\Xi_0:&
		\psset{unit=2cm,linewidth=.05cm}
		\begin{pspicture}(-2,0)(2,.7)
\psline{<-]}(-2,0)(-.6,0)
\psline{[-]}(-.6,0)(0,0)
\psline{[-]}(0,0)(.4,0)
\psline{[->}(.4,0)(2,0)
\rput(0,-.25){$0$}
\rput(.4,-.25){$\frac{1}{\alpha}$}
\rput(-.6,-.25){$\frac{1}{\alpha-r}$}
\rput(1.5,.2){$\Delta_0^0$}
\rput(.2,.2){$\Delta_0^1$}
\rput(-0.3,.2){$\Delta_0^2$}
\rput(-1.5,.2){$\Delta_0^3$}
\end{pspicture} \\
\\
\widetilde{\Xi}_0:&
		\psset{unit=2cm,linewidth=.05cm}
		\begin{pspicture}(-2,0)(2,.7)
\psline{<-]}(-2,0)(0,0)
\psline{[-]}(0,0)(.4,0)
\psline{[->}(.4,0)(2,0)
\rput(0,-.25){$0$}
\rput(.4,-.25){$\frac{1}{\alpha}$}
\rput(1.5,.2){$\widetilde{\Delta}_0^0$}
\rput(.2,.2){$\widetilde{\Delta}_0^1$}
\rput(-1.5,.2){$\widetilde{\Delta}_0^3$}
\rput(-.1,.4){$\widetilde{\Delta}_0^2$}
\psline[linewidth=.01cm](-.1,.26)(0,.13)
\end{pspicture} \\
\\
\widetilde{\Xi}_t:&
\psset{unit=2cm,linewidth=.05cm}
		\begin{pspicture}(-2,0)(2,.7)
\psline{<-]}(-2,0)(-.6,0)
\psline{[-]}(-.6,0)(0,0)
\psline{[->}(0,0)(2,0)
\rput(0,-.25){$0$}
\rput(-.6,-.25){$\frac{1}{\alpha-r}$}
\rput(1.5,.2){$\widetilde{\Delta}_t^0$}
\psline[linewidth=.01cm](0,.26)(0,.13)
\rput(0,.4){$\widetilde{\Delta}_t^1$}
\rput(-.3,.2){$\widetilde{\Delta}_t^2$}
\rput(-1.5,.2){$\widetilde{\Delta}_t^3$}
\end{pspicture} \\
\\
\widetilde{\Xi}_\infty:&
\psset{unit=2cm,linewidth=.05cm}
		\begin{pspicture}(-2,0)(2,.7)
\psline{<-]}(-2,0)(0,0)
\psline{[->}(0,0)(2,0)
\rput(0,-.25){$0$}
\rput(1.5,.2){$\widetilde{\Delta}_\infty^4$}
\rput(-1.5,.2){$\widetilde{\Delta}_\infty^5$}
\end{pspicture} \\

\end{array}
\end{displaymath}
	\end{center}
	\caption{A subdivision decomposition for $\mathcal{F}_r$}\label{fig:hirzsubdv}
\end{figure}
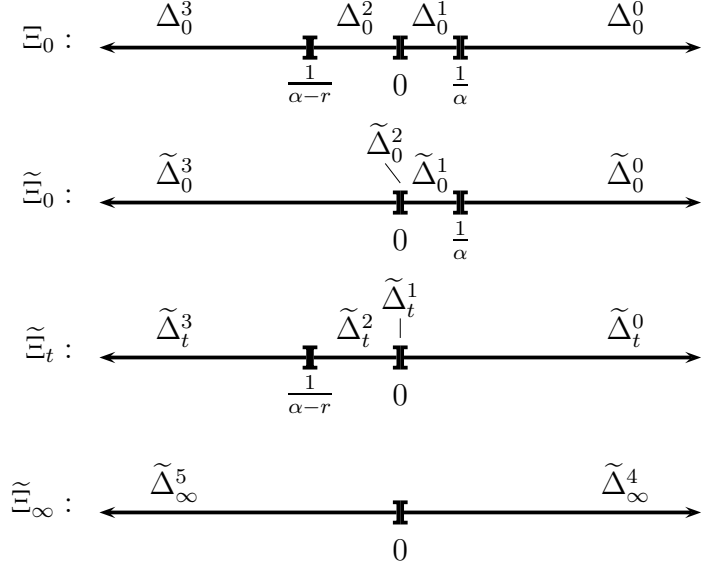

\begin{ex}
	We shall conclude with another example, namely the Hirzebruch surface $Y=\mathcal{F}_r$. Here everything is already known; if $r=1$, then $Y$ is Fano and thus rigid. Otherwise, $\mathcal{F}_r$ can be deformed to exactly those Hirzebruch surfaces $\mathcal{F}_s$ with $0\leq s< r$ and $s \equiv r \mod 2$, where by $\mathcal{F}_0$ we mean $\mathbb{P}^1\times \mathbb{P}^1$. We shall construct these deformations using subdivision decompositions, but first we turn our attention to $T_Y^1$.

First note that	$\mathcal{F}_r=\tv(\Sigma)$ where $\Sigma$ has rays $\rho_1$, $\rho_2$, $\rho_3$, and $\rho_4$ generated by $(1,0)$, $(0,1)$, $(-1,r)$, and $(0,-1)$, respectively. We immediately see that $\dim T_Y^1(u)=1$ for $u=[-\alpha,-1]$ with $r>\alpha>0$ and otherwise zero. Indeed, it is exactly for these weights $u$ that $\langle \nu(\rho_2),u\rangle =-1$, $\langle \nu(\rho_1),u\rangle <0$, and $\langle \nu(\rho_3),u\rangle <0$. In particular, we see that $\dim T_Y^1=r-1$.

Now fix some degree $-R=[-\alpha,-1]$ with $r>\alpha>0$; the fan $\Sigma$ induces a polyhedral subdivision $\Xi_0$ on the hyperplane $\langle \cdot,R\rangle = 1$. If we take $\nu(\rho_2)$ to be the origin, then $\Xi_0$ is $\mathbb{Q}$ with subdivisions at $1/(\alpha-r)$, $0$, and $1/\alpha$. Ignoring shifts with $\lambda_0$, possible subdivision decompositions are $a=(1,1,-1,-1)$ and $a=(-1,-1,1,1)$; the decomposition corresponding to $a=(1,1,-1,-1)$ is pictured in figure \ref{fig:hirzsubdv}. The images of the corresponding deformations by the Kodaira-Spencer map differ only by sign.

Using the language of T-varieties we can see what the general fiber of the above deformation is. If we take $a=(1,1,-1,-1)$, then the general fiber is the T-variety corresponding to the divisorial fan $\widetilde{\Xi}=\widetilde{\Xi}_0\otimes \{0\}+\widetilde{\Xi}_t\otimes\{t\}+\widetilde{\Xi}_\infty\otimes \{\infty\}$ on $\mathbb{P}^1$. Since the coefficient $\widetilde{\Xi}_\infty$ is trivial, this in fact corresponds to the toric variety whose fan $\widetilde{\Sigma}$ is induced by the decomposition $\widetilde{\Xi}_0$ embedded in height one and $\widetilde{\Xi}_t$ embedded in height minus one. In other words, $\widetilde{\Sigma}$ has rays through the points $(0,1)$, $(1,\alpha)$, $(0,-1)$, and $(-1,\alpha-r)$. Using a lattice automorphism, we see that this is the fan for $\mathcal{F}_{r-2\alpha}$ when $r\geq 2\alpha$ and the fan for $\mathcal{F}_{2\alpha-r}$ when $r\leq 2\alpha$. Thus, we see that we can deform $\mathcal{F}_r$ exactly to those other Hirzebruch surfaces mentioned above. Furthermore, for $0< s< r$ with $s \equiv r \mod 2$ there are exactly two degrees in which there are homogeneous deformations deforming $\mathcal{F}_r$ to $\mathcal{F}_s$, namely $[-(r+s)/2,-1]$ and $[-(r-s)/2,-1]$; if $s=0$ then the single degree in which such a deformation exists is $[-r/2,-1]$.

\end{ex}

\bibliography{tsdef}
\address{Mathematisches Institut\\
Freie Universit\"at Berlin\\
Arnimallee 3\\
14195 Berlin, Germany}{nilten@cs.uchicago.edu}

\end{document}